\documentclass{amsart} 
\usepackage{amsmath,amssymb,rawfonts,latexsym, mathtools}
\usepackage{tikz}
\usepackage{pbsi}
\usepackage[T1]{fontenc}
\usepackage{addfont}
\addfont{OT1}{rsfs10}{\rsfs}
\usetikzlibrary{matrix,arrows}
\usepackage{extarrows}
\def\test#1/{%
	\texttt{\textbackslash #1:}\qquad& A\csname #1\endcsname[\text{sub-script}]{\text{super-script}}Z}

\def\tfst #1/{%
	\texttt{\textbackslash #1:}\qquad& A\csname #1\endcsname[]{}Z}
\newtheorem{theorem}{Theorem}[section]
\newtheorem{lemma}[theorem]{Lemma}
\newtheorem{defi}[theorem]{Definition}
\newtheorem{prop}[theorem]{Proposition}
\newtheorem{coro}[theorem]{Corollary}

\newtheorem{remark}[theorem]{Remark}

\begin{document} 
\title{Closure and Interior Operators of the Category of Positive Topologies}
\author{Joaqu\'in Luna-Torres }
\thanks{Programa de Matem\'aticas, Universidad Distrital Francisco Jos\'e de Caldas,  Bogot\'a D. C., Colombia (retired professor)}

\email{joaquin.luna@sequoia-space.com}

\begin{abstract}
We define and study the notions of  closure  $\text{\rsfs{C}}$ operators  and interior $\mathbf{I}$ operators of the category   $\mathbf{CCov}$  of convergent covers which appears in positive topologies. The main motivation of this paper is to  construct the  concrete categories  $\mathbf{\text{\rsfs{C}}\text{-} CCov}$, \ of\ $\mathbf{CCov}\text{-spaces}$, and  $\mathbf{I}\text{-}\mathbf CCov$,  \ of\ $\mathbf{CCov}\text{-spaces}$ and deduce that they are  topological categories. 

\end{abstract}
\maketitle 
\baselineskip=1.7\baselineskip
\section*{0. Introduction}
Closure operators have been used intensively in Algebra (for example, G. Birkhoof, R. Pierce) and topology (for instance, K. Kuratowski, E. $\breve{C}$ech). Category theory provides a variety of notions which expand on the lattice-theoretic concept of closure operator most notably through the notion of reflective subcategory (for example, P. Freyd, J. F. Kennison, H. Herrlich). 
The notions of Grothendieck topology and Lawere-Tierney topology (see \cite{PJ}) provide standard tools in Sheaf-and Topos Theory and are most conveniently described by particular closure operators.
Both lattice-theoretic and categorical views play an important role Theoretic Computer Science.
For a topological space it is well-known that, for example, the associated closure and interior operators provide equivalent descriptions of the topology; but this is not always true in other categories, consequently it makes sense to define and study separately these operators.
The main motivation of this paper is to  construct the  concrete categories  $\mathbf{\text{\rsfs{C}}\text{-} CCov}$, \ of\ $\mathbf{CCov}\text{-spaces}$, and  $\mathbf{I}\text{-}\mathbf CCov$,  \ of\ $\mathbf{CCov}\text{-spaces}$ and deduce that they are  topological categories. 

After the construction of the coframe  $\mathbf{\mathfrak{S}_\text{{\bsifamily{cc}}}(\mathcal S)}$,  of all subobjects\linebreak $\mathcal T=(T^{c}; \lhd_{T^{c}})$ of a convergent cover  $\mathcal S=(S;\lhd)$ (in other words,  $\mathbf{\mathfrak{S}_\text{{\bsifamily{cc}}}(\mathcal S)}^{op}$ is a frame), we introduce, in section 2, the notion of closure operator $\text{\rsfs{C}}$  of the category   $\mathbf{CCov}$ as a version  of  the closure operator studied  by D. Dikranjan and W. Tholen  \cite{DT};  further, we present a  notion of closed subobjects different from the one allowed in (\cite{CS}); and finally, in that section, a reflective subcategory of $\mathbf{\mathfrak{S}_\text{{\bsifamily{cc}}}(\mathcal L)}$ is constructed.
We shall be conserned, in section 3, with a version on the category  $\mathbf CCov$ of the interior operator studied in \cite{LO}, farther we present a  notion of open subobjects at last, in that section, a co reflective subcategory of $\mathbf{\mathfrak{S}_\text{{\bsifamily{cc}}}(\mathcal L)}$ is constructed.

\section{Preliminares}
For a comprehensive account on the the categories of positive topologies we refer to F. Cirualo and G. Sambin \cite{CS}; and T. Coquand, G. Sambin, J. Smith and S. Valentin \cite{CSSV}, from whom we take the following notions:

 Formal Topology is a way to approach Topology by means of intuitionistic and
predicative tools only.
 The original definition given in \cite{SG} is now known to correspond to overt (or open) locales, in the sense that every formal topology is a predicative presentation of an overt locale and the category of formal topologies is (dually) equivalent to the full subcategory of the category of locales whose objects are overt.
 A deep rethinking of the foundations of constructive topology has led
G. Sambin to a two-sided generalization of the notion of a convergent
cover. The structure of a convergent cover can be enriched by means of a second relation, called a positivity relation, which is used to speak about some particular sub-topologies (overt weakly closed sublocales).
This enrichment produces a larger category (positive topologies) in which the category of convergent covers (locales) embeds as a reflective subcategory. The two generalizations can be combined together to obtain an extension of the category of suplattices.

\subsection{Predicative presentations of neighborhoods}

\begin{center}
\begin{minipage}[c]{0.85\textwidth}
A concrete topological space is a triple $\mathcal X \equiv (X, S,\mathbf{N})$ where $X$ is a set of concrete points, $S$ is a set of observables, $\mathbf N$ is a map from $S$ into subsets of $X$, called the {\bf neighborhood map}, which satisfies
\begin{enumerate}
\item[($B_1$)]  $X = \bigcup_{a\in S} \mathbf{N}(a)$
\item[($B_2$)]  $(\forall a, b \in S)\,\ (\forall x\in X)$
$\Big(x\  \text{\tiny{$\mathcal E$}}\ \mathbf{N}(a)\cap \mathbf{N}(b)\,\ \text{ implies} \,\ (\exists c\in S)\linebreak \big(x\ \text{\tiny{$\mathcal E$}}\ \mathbf{N}(c)\ \&\  \mathbf{N}(c)\subseteq \mathbf{N}(a)\cap \mathbf{N}(b)\big)\Big)$.
\end{enumerate}
\end{minipage}
\end{center}
Note that this description re-establishes a balance between the side of points: the concrete side, and the side of observables, or formal basic neighbourhoods, which is called the formal side.
Note that $(B_2)$ is just a rigorous writing of the usual condition stating that
if $\mathbf{N}(a)$ and $\mathbf{N}(b)$ are two neighbourhoods of $x$ then there exists a neighborhood $\mathbf{N}(c)$ of $x$ which is contained both in $\mathbf{N}(a)$ and $\mathbf{N}(b)$ and this is all what we need to obtain closure under intersection.

Now, a map $\mathbf{N} : S \rightarrow \mathcal P(X)$ is a propositional function with two arguments  $\mathbf N(x)(a)\ \mathbf{prop}\ [x : X, a : S]$, that is a binary relation, written in a more suggestive way as 
\[
x  \Vdash \mathbf{prop}\ [x : X, a : S]
\]
 and read it “x lies in a” or “x forces a”.

It is convenient to use also a few abbreviations:
\begin{align*}
 x\Vdash U & \equiv  (\exists b\ \text{\tiny{$\mathcal E$}}\ U)( x \Vdash b)\\
 ext(a)& \equiv \{ x: X\mid x\Vdash a\}\\
 ext(U) & \equiv \bigcup_{a\  \text{\tiny{$\mathcal E$}}\ U} ext(a)
 \end{align*}
 Hence $x\Vdash a$ is the same as $x\ \text{\tiny{$\mathcal E$}}\ ext(a)$ and $x\Vdash U$ is the same as $x\ \text{\tiny{$\mathcal E$}}\ ext(U)$; thus
the map $\mathbf N$ coincides with ext.
 
Then $(B_1)$ and $(B_2)$ can be rewritten as
\begin{enumerate}
\item[($B_1$)] $ (\forall x \in X)(\exists a \in S)\,\ x \Vdash a$
\item[($B_2$)] $(\forall a, b \in S)\,\  ext(a) \cap ext(b) \subseteq ext(a{\downarrow}b)$
 \end{enumerate}
where $ a{\downarrow}b \equiv \{c : S\mid ext(c) \subseteq ext(a)\ \& \ ext(c)\subseteq ext(b)\}$

\subsection{Basic cover}
In  \cite{BS}, the authors show that Sup-lattices can be characterized as pairs ($L;\bigvee)$ where $\bigvee$ is an infinitary operation on $L$ satisfying 
\begin{enumerate}
\item[(i)] $\bigvee\{x\}=x$ for every $x\in L$;
\item[(ii)] $\bigvee_{i\in I}(\bigvee U_{i}) = \bigvee (\bigcup_{i\in I} U_{i})$ for every family $(U_i)_{i\in I}$ of subsets of $L$.
\end{enumerate}
Now, they define $x\leq y$ putting $\bigvee\{x, y\} =y$.
This partial order induce a relation $\prec$ between subsets, where the intended meaning of $U\prec  W$ is that $\bigvee U\leq \bigvee W$. Recalling that $a =\bigvee\{a\}$, the characterizing property of joins can be written in terms of $\prec$ as
$(\forall a \in U)(\{a\} \prec V)$ iff $U\prec V$. For  any  preorder $\prec$, they  define a relation between elements and subsets by putting
$$ a \lhd  U \equiv \{a\}\prec U.$$

In general, the opposite of a set-based suplattice need not be set-based. In the set-based case all the information about the suplattice under consideration can be coded by means of a cover relation on the base:
\begin{center}
\begin{minipage}[c]{0.85\textwidth}
 Let $S$ be a set. A small relation $\lhd$ between elements and subsets of $S$ is called a (basic) {\bf cover} if
\begin{enumerate}
\item  $a\  \text{\tiny{$\mathcal E$}}\ U \Rightarrow a \lhd U$
\item $ a \lhd U\ \&\ (\forall u \ \text{\tiny{$\mathcal E$}}\ U)(u \lhd V ) \Rightarrow a \lhd V$
for every $a\in S$ and $U; V \subseteq S$.
\end{enumerate}
\end{minipage}
\end{center}
The motivating example is given by a set-based suplattice with base $S$, where
$a\lhd U$ is taken to mean $a\leqslant \bigvee U$.

 A base for the suplattice (least upper bound lattice) $(L;\leq)$ is a set $S \subseteq L$ such that
$\bigvee \{a\in S\mid a\leq p\}= p$  for all $p$ in $L$. This is called  a set-based suplattice. 

A basic cover $(S;\lhd)$ has to be understood as a presentation of a set-based suplattice, where $S$ plays the role of a set of codes for the base. Indeed, any cover $(S;\lhd)$ can be extended to a preorder $U \lhd V$ on $\mathcal P(S)$ defined by $(\forall u \ \text{\tiny{$\mathcal E$}}\ U)(u \lhd V )$. This induces an equivalence relation $=_{\lhd}$ on $\mathcal P(S)$  where
$ U =_{\lhd} V$ is $U \lhd V\  \& \ V\lhd U$. The quotient collection $\mathcal P(S)_{/=_{\lhd}}$ is a suplattice with $\bigvee_{i} [U_{i}] = [\bigcup_{i} U_{i}]$ (and $[U] \leq [V]$
if and only if $U\lhd  V$). Such a suplattice has a base, namely the set $\{ [a]\mid a\in S\}$. (Here we have adopted a convention we are going to use quite often: for readability's sake, we denote a singleton by its unique element.)

To complete the  illustration, one should note that: 
\begin{enumerate}
\item[($i$)] the cover induced by a set-based suplattice L presents a suplattice which is isomorphic to L, the isomorphism being given by the two mappings $x\mapsto \{a\in S \mid a\leq x\}$ and $[U]\mapsto \bigvee U$;
\item[($ii$)] the cover associated to the suplattice presented by a cover $(S;\lhd)$ is isomorphic to $(S;\lhd)$ itself, according to the definition of morphism given below. Note that each set-based suplattice can be presented by several covers; all of them are going to be isomorphic to each other, according to the notion of morphism we are going to introduce below.
\end{enumerate}

\subsection{Morphisms between basic covers}
 Let $\mathcal S_1 = (S_1;\lhd_1)$ and $\mathcal S_2 = (S_2;\lhd_2)$ be two basic covers. A small relation\,\,\ $\text{\Large \bsifamily{s}} \subseteq S_1 \times S_2$ {\bf respects the covers} if
\[
U \lhd_2 V \Rightarrow \text{\Large \bsifamily{s}}^{-}U \lhd_1 \text{\Large \bsifamily{s}}^{-}V\,\ \text{for all}\,\ U; V \subseteq S_2
\]
where $\text{\Large \bsifamily{s}}^{-}W = \{ a \in S_1 \mid (\exists w\in W)(a\ \text{\Large \bsifamily{s}}\ w)\}$.
A morphism between $\mathcal S_1$ and $\mathcal S_2$ is an equivalence class of relations between $S_1$ and $S_2$ which respect the covers, where two relations $\text{\Large \bsifamily{s}}$ and $\text{\Large \bsifamily{s'}}$ are equivalent if $\text{\Large \bsifamily{s}}^{-}W =_{\lhd_1} \text{\Large \bsifamily{s'}}^{\ -}W$ for every $W \subseteq S_2$.

The previous  definition has a very natural meaning: a morphism between two covers is just a presentation of a suplattice homomorphism between the corresponding suplattices.

Basic covers and their morphisms form a category, called {\bf BCov}, which is
dual to the category {\bf SL} of suplattices, impredicatively. The
previous discussion says that
\[
\mathbf{BCov}\big((S_1;\lhd_1); (S_2;\lhd_2)\big)= \mathbf{SL}\big(\mathcal P(S_2)_{/=_{\lhd_2}};\mathcal P(S_1)_{/=_{\lhd_1}} \big).
\]
\subsection{Convergent basic cover}
A basic cover is {\bf convergent} if its corresponding suplattice is a frame.
A morphism between convergent covers is a morphism of basic covers whose corresponding suplattice homomorphism is, in fact, a frame homomorphism (preserves finite meets). The resulting category will be called $\mathbf{CCov}$. Impredicatively, $\mathbf{CCov}$ is dual to the category Frm of frames and hence equivalent to the category $\mathbf Loc$ of locales. 

An explicit description of convergent covers and their morphisms is the following:
\begin{center}
\begin{minipage}[c]{0.85\textwidth}
  {\it A basic cover $(S;\lhd)$ is convergent if and only if
\begin{itemize}
\item[$\star$] $a \lhd U \& a \lhd  V \Rightarrow a \lhd U{\downarrow}V$ for every $a\in S$ and $U; V \subseteq S$.
\end{itemize}
where $U{\downarrow}V = \{ b \in S\mid b \lhd u\ \& \ b \lhd v\ \text{for some}\ (u; v) \in U \times  V\}$. In this case,
$[U] \wedge [W] = [U{\downarrow}W]$.}
\end{minipage}
\end{center}

\begin{center}
\begin{minipage}[c]{0.85\textwidth}
A morphism $\text{\Large \bsifamily{s}} : (S_1;\lhd_1) \rightarrow (S_2;\lhd_2)$ between convergent covers is a morphism of basic covers such that
\begin{itemize}
\item $ S_1 \lhd_1 \text{\Large \bsifamily{s}}^{-}S_2$ and
\item $(\text{\Large \bsifamily{s}}^{-}U){\downarrow}_1(\text{\Large \bsifamily{s}}^{-}V) \lhd_1 \text{\Large \bsifamily{s}}^{-}(U{\downarrow}_2 V)$ for every $U; V \subseteq S_2$.
\end{itemize}
\end{minipage}
\end{center}
\subsection{Basic and positive topology}
It is convenient to use the symbol $\between$ for inhabited intersection, that is,
$U\between V \xLongleftrightarrow {def} (\exists a\in S)(a\in U\ \&\ a\in V)$  for $U; V\subseteq S$.

An element $x$ of a locale $L$ is {\bf positive} if ($x \leq \bigvee Y) \Rightarrow (Y \between L)$ for every $Y\subseteq L$. With classical logic, $x$ is positive if and only if $x\ne 0$. In the language of formal topology this notion is translated as follows, which requires some impredicativity: 
\begin{center}
\begin{minipage}[c]{0.85\textwidth}
 Given a (convergent) cover $(S;\lhd)$,  $a\in S$ is said to be positive if $(a \lhd U)\rightarrow (U \between S)$ for every $U \subseteq S$.  {\bf Pos} is the subset of positive elements of $S$. A subset $U \subseteq S$ is said to be positive if $U \between {\mathbf{Pos}}$.

A convergent cover $(S;\lhd)$ is {\bf overt} if $a \lhd \{a\} \cap {\mathbf{Pos}}$ for every $a\in S$.
\end{minipage}
\end{center}
$(S;\lhd)$ is overt if and only if $[U] = [U \cap {\mathbf{Pos}}]$ for every $U \subseteq S$.

Overt locales are usually defined in an equivalent way, as follows:
The category of locales has a terminal object which, as a frame, is the power
$\mathcal P(1)$ of the singleton $1 = \{0\}$. This corresponds to the convergent cover $(1;\in $). It can be thought of  the elements of $\mathcal P(1)$ as propositions modulo logical equivalence (that is, truth values).

For each convergent cover $(S;\lhd)$ there exists a unique (up to equivalence)
morphism $\text{\Large \bsifamily{s}} :(S;\lhd) \rightarrow (1;\in)$ between convergent covers (put $\text{\Large \bsifamily{s}}^{-}0 = S$). As a frame homomorphism $\mathcal P(1)\rightarrow \mathcal P(S)_{/=_{\lhd}}$ it maps a proposition $p$ to the equivalence class $[\{a \in S \mid p\}]$.

A basic cover $(S;\lhd)$ equipped with a compatible positivity relation  is called a {\bf basic topology}. A convergent cover equipped with a compatible positivity relation is called a {\bf positive topology}.

\section{Closure Operators}

In this section we shall be conserned with a version on the category  $\mathbf{CCov}$ of  the closure operator studied  by D. Dikranjan and W. Tholen  \cite{DT}.

It is important to remember that impredicatively, $\mathbf{CCov}$ is dual to the category $\mathbf{Frm}$ of frames and hence equivalent to the category $\mathbf{Loc}$ of locales. 

Our first aim in this section is to describe  a special class of subobjects of a set $S$ in such a way that they will be sublocales of $\mathcal P(S)$, and with them  we shall get a coframe. 
\begin{lemma}
Let $S$ be a  set and let $T$ be a subset of $S$. then 
\[
\mathcal P_{*}(T) = \{V\cup T^{c}\mid V\in \mathcal P(T)\}\,\ \text{ is a sublocale of} \,\ \mathcal P(S). 
\]

\end{lemma}
( Here $T^{c}$ is the complement of $T$ in $S$)
\begin{proof}
Clearly $ \mathcal P_{*}(T)$ is s a complete lattice  satisfying the distributivity law of arbitary joins and finite meets, and  for every $V\cup T^{c}\in \mathcal P_ {*}(T)$ and every $U\in \mathcal P(S)$,\,\ $(U^{c}\cap T^{c})\cup V\in \mathcal P(T)$ and $U^{c}\cap T^{c}\subseteq T^{c}$, therefore $U\rightarrow (V\cup T^{c})$ is an element of $ \mathcal P_{*}(T)$\ (here \ ``$\rightarrow$''\ denote the Heyting implication).
\end{proof}

It is clear that $\bigwedge \mathcal P_{*}(T) = T^{c}$, therefore it is important to use $T^{c}$ in the definition of subobject
\footnote{ Since $U\cup \mathfrak V^{c} \in \mathcal P_{*}(\mathfrak V)$ for every $U\in \mathcal P(S)$,\,\ $(\mathfrak V, \lhd_{\mathfrak U })$ is a convergent coverage for each $\mathfrak V\subseteq S$, provided that $(S;\lhd)$ is convergent}.

\begin{defi}
A subobject $\mathcal T=(T; \lhd_{T})$ of a convergentc cover   $\mathcal S=(S;\lhd)$  consists of
\begin{enumerate}
\item The complement $T^{c}$ of a subset $T$ of $S$;
\item a convergent cover $\lhd_{T^{c}} \subseteq T\times \mathcal P_{*}(T)$ obtained as follows: if $a\in T^{c}$ and $a \lhd U$ for $U\subseteq S$,then $a\lhd_{T^{c}} (U\cup T^{c})$\footnote{ Since  $U\cup T^{c} \in \mathcal P_{*}(T)$ for every $U\in \mathcal P(S)$,\,\  $(T, \lhd_{T})$ is a convergent cover for every $T\subseteq S$, whenever $(S;\lhd)$ is convergent.} .
\end{enumerate}
\end{defi}
Note that since $U\cup T^{c}= (U\cap T)\cup(U\cap T^{c})\subseteq (U\cap T)\cup U^{c}$,  then $(U\cup T^{c}) \in  \mathcal P_{*}(T)$.

From now on, we shall denote by $\mathbf{\mathfrak{S}_\text{{\bsifamily{cc}}}(\mathcal S)}$ the coframe  of all subobjects\linebreak $\mathcal T=(T^{c}; \lhd_{T^{c}})$ of  $\mathcal S=(S;\lhd)$(i.e.  $\mathbf{\mathfrak{S}_\text{{\bsifamily{cc}}}(\mathcal S)}^{op}$ is a frame).
 \begin{defi}
 A closure operator $\text{\rsfs{C}}$  of the category   $\mathbf{CCov}$ is given by a family $\text{\rsfs{C}} =(c_{\text{\tiny{$\mathcal S$}}})_{\text{$ \mathcal S\in \mathbf{CCov}$}}$ of maps $c_{\text{\tiny{$\mathcal S$}}}:\mathfrak{S}_\text{{\bsifamily{cc}}}(S)\rightarrow \mathfrak{\mathcal S}_\text{{\bsifamily{cc}}}(\mathcal S)$ such that
 \begin{itemize}
 \item[($C_1$)] $\left(\text{Extension}\right)$\,\ $ \mathcal T\subseteq c_{\text{\tiny{$\mathcal S$}}}(\mathcal T) $ for all $\mathcal T \in \mathcal{S}_\text{{\bsifamily{cc}}}(S)$;
 \item[($C_2$)] $\left(\text{Monotonicity}\right)$\,\  If $\mathcal U\subseteq \mathcal T$ in $\mathcal{S}_\text{{\bsifamily{cc}}}(S)$, then $c_{\text{\tiny{$\mathcal S$}}}(\mathcal U)\subseteq c_{\text{\tiny{$\mathcal S$}}}(\mathcal T)$
 \item[($C_3$)] $\left(\text{Lower bound}\right)$\,\  $c_{\text{\tiny{$\mathcal S$}}}(\Phi)=\Phi$, where $\Phi =(\emptyset, \lhd_{\emptyset})$.
 \end{itemize}
 \end{defi}

 \begin{defi}
 An $\text{\rsfs{C}}$-space is a pair $(\mathcal S, c_{\text{\tiny{$\mathcal S$}}})$ where $\mathcal S$ is an object of  $\mathbf{CCov}$   and $ c_{\text{\tiny{$\mathcal S$}}}$ is a closure map on $\mathcal S$.
 \end{defi}

 \begin{remark}
 If $\text{\Large \bsifamily{s}} \subseteq S_1 \times S_2$ is a representative of   an equivalence class of relations between $\mathcal S_1=(S_1,\lhd_1)$ and $\mathcal S_2=(S_2,\lhd_2)$ which respect the covers, we donte by
 \begin{itemize}
\item  $\text{\Large \bsifamily{s}}^{-1}$ the inverse of $\text{\Large \bsifamily{s}}$;
\item $\text{\Large \bsifamily{s}}^{\rightarrow}$ the direct image of $\text{\Large \bsifamily{s}}$ defined by $$(\forall X\subset S_{1})( \text{\Large \bsifamily{s}}^{\rightarrow}(X)=\big\{t\in S_{2}\mid \big (\exists s\in X\big)\big( (s,t)\in \text{\Large \bsifamily{s}}\big)\big\}; $$
\item $\text{\Large \bsifamily{s}}^{\leftarrow}$ the inverse  image of $\text{\Large \bsifamily{s}}$ is the direct image of $\text{\Large \bsifamily{s}}^{-1}$.
 \end{itemize}
 \end{remark}
 
 \begin{defi}
 A morphism $\text{\Large \bsifamily{s}}:\mathcal L\rightarrow \mathcal S $ of $\mathbf  CCov$ is said to be $\text{\rsfs{C}}$-continuous if

 \begin{equation}\label{c-conti}
 \text{\Large \bsifamily{s}}^{\rightarrow}\big( c_{\text{\tiny{$\mathcal L$}}}(\mathcal T)\big) \subseteq c_{\text{\tiny{$\mathcal S$}}}\big(  \text{\Large \bsifamily{s}}^{\rightarrow}(\mathcal T)\big)
 \end{equation}
 for all $\mathcal T \in \mathfrak{S}_\text{{\bsifamily{cc}}}(\mathcal L)$.
 \end{defi}
  
 Note that, in the presence of requirement ($C_2$), the continuity condition can equivalently be expresssed as
 \begin{equation}\label{equi-conti}
 c_{\text{\tiny{$\mathcal L$}}}(\text{\Large \bsifamily{s}}^{\leftarrow}(\mathcal U)) \subseteq \text{\Large \bsifamily{s}}^{\leftarrow}\big[c_{\text{\tiny{$\mathcal S$}}}(\mathcal U)\big]
 \end{equation}
 for all $\mathcal U \in \mathfrak{S}_\text{{\bsifamily{cc}}}(\mathcal  S)$.
 Indeed, from (\ref{c-conti}), we have 
 \[
  \text{\Large \bsifamily{s}}^{\rightarrow}\big(c_{\text{\tiny{$\mathcal L$}}}( \text{\Large \bsifamily{s}}^{\leftarrow}(\mathcal U)\big) \subseteq c_{\text{\tiny{$\mathcal S$}}}\left( \text{\Large \bsifamily{s}}^{\rightarrow}( \text{\Large \bsifamily{s}}^{\leftarrow}(\mathcal U)\right)\subseteq c_{\text{\tiny{$S$}}}(\mathcal U).
\]  

 consequently, $ c_{\text{\tiny{$\mathcal L$}}}(\text{\Large \bsifamily{s}}^{\leftarrow}(\mathcal U)) \subseteq \text{\Large \bsifamily{s}}^{\leftarrow}\big[c_{\text{\tiny{$\mathcal S$}}}(\mathcal U)\big]$.
 \begin{prop}
 Let $\text{\Large\bsifamily{s}}:\mathcal L\rightarrow \mathcal M$ and $\text{\Large\bsifamily{t}}:\mathcal M\rightarrow\mathcal N$ be two $\text{\rsfs{C}}$-continuous morphisms  of $\mathbf  CCov$ then \ $\text{\Large\bsifamily{t $\centerdot$ s}}$ is an $\text{\rsfs{C}}$-continuous morphism of  $\mathbf{CCov}$.
 \end{prop}
 \begin{proof}
 Since $\text{\Large\bsifamily{s}}:\mathcal L\rightarrow \mathcal M$ is $\text{\rsfs{C}}$-continuous, we have  $$\text{\Large\bsifamily{s}}^{\rightarrow}[c_{\text{\tiny{$\mathcal L$}}}(\mathcal T)] \subseteq c_{\text{\tiny{$\mathcal M$}}}\left( \text{\Large\bsifamily{s}}^{\rightarrow}(\mathcal T\right)$$
  for all $\mathcal T\in \mathfrak{S}_\text{{\bsifamily{cc}}}(\mathcal S)$, it fallows that
  $$\text{\Large\bsifamily{t}}^{\rightarrow}\big[\text{\Large\bsifamily{s}}^{\rightarrow}[c_{\text{\tiny{$\mathcal L$}}}(\mathcal T)]\big] \subseteq \text{\Large\bsifamily{t}}^{\rightarrow}\big[c_{\text{\tiny{$\mathcal M$}}}\left(\text{\Large\bsifamily{s}}^{\rightarrow}(\mathcal T\right)\big]$$
  now,  by the $\text{\rsfs{C}}$-continuity of $\text{\Large\bsifamily{t}}$,
 $$ \text{\Large\bsifamily{t}}^{\rightarrow}\big[c_{\text{\tiny{$\mathcal M$}}}\left( \text{\Large\bsifamily{s}}^{\rightarrow}(\mathcal T)\right)\big] \subseteq c_{\text{\tiny{$\mathcal N$}}}\left( \text{\Large\bsifamily{t}}^{\rightarrow}\big[\text{\Large\bsifamily{s}}^{\rightarrow}(\mathcal T)\big]\right)$$ 
 
therefore $$\text{\Large\bsifamily{t}}^{\rightarrow}\big[\text{\Large\bsifamily{s}}^{\rightarrow}[c_{\text{\tiny{$\mathcal L$}}}(\mathcal T)\big] \subseteq c_{\text{\tiny{$\mathcal N$}}}\left( \text{\Large\bsifamily{t}}^{\rightarrow}\big[\text{\Large\bsifamily{s}}^{\rightarrow}(\mathcal T)\big]\right),$$
  that is to say 
  $$(\text{\Large\bsifamily{t $\centerdot$ s}})^{\rightarrow}\big[ c_{\text{\tiny{$\mathcal L$}}}(\mathbf T)\Big]\subseteq c_{\text{\tiny{$\mathcal N$}}}\Big( (\text{\Large\bsifamily{t $\centerdot$s}})^{\rightarrow} (\mathbf T)\Big)$$
\end{proof}
  As a consequence we obtain

 \begin{defi}
 The category\ $\mathbf{\text{\rsfs{C}}\text{-} CCov}$\ of\ $\mathbf{CCov}-\text{spaces}$ comprises the following data:
 \begin{enumerate}
 \item {\bf Objects}: Pairs $(\mathcal S,c_{\text{\tiny{$S$}}})$ where $\mathcal S=(S,\lhd)$ is an object of $\mathbf{CCov}$ and $c_{\text{\tiny{$\mathcal S$}}}$ is a closure map on $\mathcal S$.
\item {\bf Morphisms}: Morphisms of $\mathbf{CCov}$ which are $\text{\rsfs{C}}$-continuous.
 \end{enumerate}
 \end{defi}
 
 \subsection{The lattice structure of all closure operators}
 For the category $\mathbf{CCov}$ we consider the collection
 \[
 \text{\rsfs C} l(\mathbf{CCov})
 \]
 of all  closure operators on $\mathbf{CCov}$. It is ordered by
\[
\text{\rsfs C}\leqslant\text{\rsfs D}\Leftrightarrow c_\text{\small{\text{\tiny{$\mathcal S$}}}}(\mathcal T)\subseteq d_{\text{\tiny{$\mathcal S$}}}(\mathcal T), 
\]
for all $ \mathcal T\in \mathbf{\mathfrak{S}_\text{{\bsifamily{cc}}}(\mathcal S)}$   and for all $ \mathcal S \,\ \text{ object of}\ \mathbf{CCov}$

This way $\text{\rsfs C} l(\mathbf{CCov})$ inherits a lattice structure from $\mathbf{\mathfrak{S}_\text{{\bsifamily{cc}}}(\mathcal S)}$:

 \begin{prop}
 Every family $\text{\rsfs C}(_{\text{\tiny{$\lambda$}}})_{\text{\tiny{$\lambda\in \Lambda$}}}$ in $ \text{\rsfs C} l(\mathbf{CCov})$ has a join $\bigvee\limits_{\text{\tiny{$\lambda\in \Lambda $}}}{\text{\rsfs C}}_{\text{\tiny{$\lambda $}}}$ and a meet $\bigwedge\limits_{\text{\tiny{$\lambda\in \Lambda $}}}{\text{\rsfs C}}_{\text{\tiny{$\lambda $}}}$ in $ \text{\rsfs C} l(\mathbf{CCov})$. The discrete closure operator
 \[ {\text{\rsfs C}}_{\text{\tiny{$D$}}}=({c_{\text{\tiny{$D$}}}}_{\text{\tiny{$\mathcal L$}}})_{\text{$\mathcal L\in \mathbf{CCov}$}}\,\,\ \text{with}\,\,\ {c_{\text{\tiny{$D$}}}}_{\text{\tiny{$L$}}}(\mathcal T)=\mathcal T\,\,\ \text{for all}\,\ \mathcal T\in \mathbf{\mathfrak{S}_\text{{\bsifamily{cc}}}(\mathcal L)}
 \]
  is the least element in $ \text{\rsfs C} l(\mathbf{CCov})$, and the trivial closure operator
  \[ 
{\text{\rsfs C}}_{\text{\tiny{$T$}}}=({c_{\text{\tiny{$T$}}}}_\text{\tiny{$\mathcal L$}})_{\text{$\mathcal L\in \mathbf{CCov}$}}\,\,\ \text{with}\,\,\ {c_{\text{\tiny{$T$}}}}_{\text{\tiny{$\mathcal L$}}}(\mathcal T)=
  \begin{cases}
\mathcal L& \text{for all}\,\ \mathcal T\in \mathbf{\mathfrak{S}_\text{{\bsifamily{cc}}}(\mathcal L)},\,\ S\ne \mathbf \Phi\\
\mathbf \Phi&\text {if}\,\ \mathcal T=\mathbf \Phi
\end{cases}
 \]
is the largest one.
\end{prop}
\begin{proof}
For $\Lambda\ne\emptyset$, let $\widehat{\text{\rsfs C}}=\bigvee\limits_{\text{\tiny{$\lambda\in\Lambda $}}}{\text{\rsfs C}}_{\text{\tiny{$\lambda $}}}$, then 
 \[
 \widehat{c_{\text{\tiny{$\mathcal L$}}}}=\bigvee\limits_{\text{\tiny{$\lambda\in \Lambda$}}} {c_{\text{\tiny{$\lambda $}}}}_{\text{\tiny{$\mathcal L$}}},
 \]
 for all  $\mathcal L$ object of $\mathbf{CCov}$, satisfies
\begin{itemize}
 \item $ \mathcal T \subseteq \widehat{c_{\text{\tiny{$\mathcal L$}}}}(\mathcal T)$,  because $ \mathcal T \subseteq {c_{\text{\tiny{$\lambda $}}}}_{\text{\tiny{$\mathcal T$}}}(\mathcal T)$ for all $\mathcal T\in \mathbf{\mathfrak{S}_\text{{\bsifamily{cc}}}(\mathcal L)}$ and for all $\lambda \in \Lambda$.
\item If $\mathcal R \leqslant \mathcal T$ in $\mathbf{\mathfrak{S}_\text{{\bsifamily{cc}}}(\mathcal L)}$ then ${c_{\text{\tiny{$\lambda $}}}}_{\text{\tiny{$\mathcal L$}}}(\mathcal R)\subseteq {c_{\text{\tiny{$\lambda $}}}}_{\text{\tiny{$\mathcal L$}}}(\mathcal T)$   for all $\lambda \in \Lambda$, therefore $ \widehat{c_{\text{\tiny{$\mathcal L$}}}}(\mathcal R)\subseteq \widehat{c_{\text{\tiny{$\mathcal L$}}}}(\mathcal T)$.
\item Since ${c_{\text{\tiny{$\lambda $}}}}_{\text{\tiny{$\mathcal L$}}}(\Phi))=\Phi $  for all $\lambda \in \Lambda$, we have that  $ \widehat{c_{\text{\tiny{$\mathcal L$}}}}(\Phi)=\Phi$.
 \end{itemize}
Similarly  $\bigwedge\limits_{\text{\tiny{$\lambda\in \Lambda $}}}{\text{\rsfs C}}_{\text{\tiny{$\lambda $}}}$,\,\  $ {\text{\rsfs C}}_{\text{\tiny{$D$}}}$ and 
${\text{\rsfs C}}_{\text{\tiny{$T$}}}$ are closure operators.
 \end{proof}
 \begin{coro}\label{c-complete}
  For  every  object $\mathcal L$ of $\mathbf{CCov}$
  \[
 \text{\rsfs C} l(\mathcal L) = \{c_{\text{\tiny{$\mathcal L$}}}\mid c_{\text{\tiny{$\mathcal L$}}}\,\ \text{ is a closure map on}\,\ \mathcal L\}
  \]
  is a complete lattice.
 \end{coro}
  
\subsection{Initial closure operators}
Let \ $\mathbf{\text{\rsfs{C}}\text{-} CCov}$\ be the category of\ $\mathbf{CCov}$-spaces. 
Let  $(\mathcal M,c_{\text{\tiny{$\mathcal M$}}})$ be an object\linebreak of 
$\mathbf{\text{\rsfs{C}}\text{-} CCov}$ and let $\mathcal L$ be an object of $\mathbf{CCov}$. 
For each morphisms 
$\text{\Large \bsifamily{s}}:\mathcal L\rightarrow \mathcal M $ of $\mathbf{CCov}$ 
we define on $\mathcal L$ the map
\begin{equation} \label{c-initial}
c_{\text{\tiny{$\mathcal{L}_{\text{\large{\bsifamily{s}}}}$}}}:= \text{\Large \bsifamily{s}}^{\leftarrow}\centerdot c_{\text{\tiny{$\mathcal M$}}}\centerdot\text{\Large \bsifamily{s}}^{\rightarrow}
\end{equation}
\begin{center}
\begin{tikzpicture}[scale=0.8]
\node (A) at (0,3) {$\mathfrak{A}_\text{{\bsifamily{cc}}}(\mathcal L)$};
\node (B) at (3,3) {$\mathfrak{A}_\text{{\bsifamily{cc}}}(\mathcal M)$};
\node (C) at (0,0) {$\mathfrak{A}_\text{{\bsifamily{cc}}}(\mathcal L)$};
\node (D) at (3,0) {$\mathfrak{A}_\text{{\bsifamily{cc}}}(\mathcal M)$};
\draw[->] (A) -- (B);
\draw[->, dashed] (A) -- (C);
\draw[->] (B) -- (D);
\draw[->] (D) -- (C);
\node at (1.5,3.3) {$ \text{\large \bsifamily{s}}^{\rightarrow} $};
\node at (1.5,-0.3) {$ \text{\large \bsifamily{s}}^{\leftarrow} $};
\node at (3.4,1.5) {$ c_{\text{\tiny{$\mathcal M$}}} $};
\node at (-0.4,1.5) {$ c_{\text{\tiny{${\mathcal L}_{ \text{\tiny \bsifamily{s}}}$}}}$};
\end{tikzpicture}
\end{center}
\begin{prop}\label{c-ini-cont}
Equation (\ref{c-initial}) define a  map of the closure operator $\text{\rsfs C}$ for which the morphism  $\text{\Large \bsifamily{s}}:\mathcal L\rightarrow \mathcal M$ in $\mathbf{CCov}$ is $\text{\rsfs C}$-continuous.
\end{prop}
\begin{proof}\ 
\begin{enumerate}
\item[($C_1)$] $\left(\text{Extension}\right)$\,\ Let  $\mathcal L$ be in $\mathbf{CCov}$, then since $ c_{\text{\tiny{$\mathcal M$}}}$ is a closure map on $\mathcal M$ it follows that $\text{\Large \bsifamily{s}}^{\rightarrow}(\mathcal T)\subseteq c_{\text{\tiny{$\mathcal M$}}}\big( \text{\Large \bsifamily{s}}^{\rightarrow}(\mathcal T)\big)$, which is equivalent to saying that $\mathcal T\subseteq (\text{\Large \bsifamily{s}}^{\leftarrow}\centerdot c_{\text{\tiny{$\mathcal M$}}}\centerdot \text{\Large \bsifamily{s}}^{\rightarrow})(\mathcal T)= c_{\text{\tiny{$\mathcal{L}_{\text{\large{\bsifamily{s}}}}$}}}(\mathcal T)$;
 \item[($C_2)$] $\left(\text{Monotonicity}\right)$\,\   $\mathcal R\subseteq\mathcal T$ in $ \mathfrak{S}_\text{{\bsifamily{cc}}}(\mathcal L)$, implies $ \text{\Large \bsifamily{s}}^{\rightarrow}(\mathcal R)\subseteq \text{\Large \bsifamily{s}}^{\rightarrow}(\mathcal T)$, then we have
 $$(c_{\text{\tiny{$\mathcal M$}}}\centerdot \text{\Large \bsifamily{s}}^{\rightarrow})(\mathcal R)\subseteq (c_{\text{\tiny{$\mathcal M$}}}\centerdot\text{\Large \bsifamily{s}}^{\rightarrow})(\mathcal T),$$
 consequently
 $$(\text{\Large \bsifamily{s}}^{\leftarrow}\centerdot c_{\text{\tiny{$\mathcal M$}}}\centerdot \text{\Large \bsifamily{s}}^{\rightarrow})(\mathcal R)\subseteq (\text{\Large \bsifamily{s}}^{\leftarrow}\centerdot c_{\text{\tiny{$\mathcal M$}}}\centerdot\text{\Large \bsifamily{s}}^{\rightarrow})(\mathcal T).$$
\item[($C_3$)]  $\left(\text{Lower bound}\right)$\,\
$(\text{\Large \bsifamily{s}}^{\leftarrow}\centerdot c_{\text{\tiny{$\mathcal M$}}}\centerdot \text{\Large \bsifamily{s}}^{\rightarrow})(\Phi)\subseteq (\text{\Large \bsifamily{s}}^{\leftarrow}\centerdot c_{\text{\tiny{$\mathcal M$}}}\centerdot\text{\Large \bsifamily{s}}^{\rightarrow})(\Phi).$
\end{enumerate}
Finally, 
 
\[
 \text{\Large \bsifamily{s}}^{\rightarrow}\big(c_{\text{\tiny{$\mathcal{L}_{\text{\large{\bsifamily{s}}}}$}}}(\mathcal T)  \big)=\text{\Large \bsifamily{s}}^{\rightarrow}\big(\text{\Large \bsifamily{s}}^{\leftarrow}\centerdot c_{\text{\tiny{$\mathcal M$}}}\centerdot\text{\Large \bsifamily{s}}^{\rightarrow}(\mathcal T)\big)\subseteq c_{\text{\tiny{$\mathcal M$}}}\big(  \text{\Large \bsifamily{s}}^{\rightarrow}(\mathcal T)\big)
\]
 for all $\mathcal T \in \mathfrak{S}_\text{{\bsifamily{cc}}}(\mathcal L)$.
 
\end{proof}
It is clear that $c_{\text{\tiny{$\mathcal{L}_{\text{\large{\bsifamily{s}}}}$}}}$ is the coarsest closure map on $\mathcal L$ for which the morphism $\text{\Large \bsifamily{s}}$ is $\text{\rsfs{C}}$-continuous; more precisaly

\begin{prop}\label{c-unique}

Let $(\mathcal M,  c_{\text{\tiny{$\mathcal M$}}})$ and $(\mathcal N,  c_{\text{\tiny{$\mathcal N$}}})$  be objects of $\mathbf{\text{\rsfs{C}}\text{-} CCov}$, and let $\mathcal L$ be an object of  $\mathbf{CCov}$. For each morphism  $\text{\Large \bsifamily{t}}:\mathcal N\rightarrow \mathcal L$ in  $\mathbf{CCov}$ and for the 
 $\text{\rsfs{C}}$-continuous morphism $\text{\Large \bsifamily{s}}:(\mathcal L, c_{\text{\tiny{$\mathcal{L}_{\text{\large{\bsifamily{s}}}}$}}})\rightarrow (\mathcal  M, c_{\text{\tiny{$\mathcal M$}}})$, \ $\text{\Large \bsifamily{t}}$  is $\text{\rsfs{C}}$-continuous if and only if $\text{\Large \bsifamily{t}}\centerdot \text{\Large \bsifamily{s}}$ is $\text{\rsfs{C}}$-continuous.
\end{prop}
\begin{proof}
Suppose that $\text{\Large \bsifamily{s}}\centerdot \text{\Large \bsifamily{t}}$ is $\text{\rsfs{C}}$-continuous, i. e.
$$
(\text{\Large \bsifamily{s}}\centerdot \text{\Large \bsifamily{t}})^{\rightarrow}\big (c_{\text{\tiny{$\mathcal N$}}}(\mathcal U)\big) \subseteq c_{\text{\tiny{$\mathcal M$}}}\left( (\text{\Large \bsifamily{s}}\centerdot \text{\Large \bsifamily{t}})^{\rightarrow} (\mathcal U)\right)
$$
for all $\mathcal U \in \mathfrak{S}_\text{{\bsifamily{cc}}}(\mathcal N)$.

It follows that 
\begin{align*}
 \text{\Large \bsifamily{t}}^{\rightarrow}[c_ {\text{\tiny{$\mathcal N$}}}(\mathcal U)]&\subseteq \text{\Large \bsifamily{s}}^{\leftarrow}\big[c_ {\text{\tiny{$\mathcal M$}}}\big(\text{\Large \bsifamily{s}}^{\rightarrow}\big(\text{\Large \bsifamily{t}}^{\rightarrow}(\mathcal U)\big)]\\
  &\subseteq \big( \text{\Large \bsifamily{s}}^{\leftarrow}\cdot c_ {\text{\tiny{$\mathcal M$}}}\cdot\text{\Large \bsifamily{s}}^{\rightarrow}\big)(\text{\Large \bsifamily{t}}^{\rightarrow}(\mathcal U)\\
 & =c_{\text{\tiny{$\mathcal{L}_{\text{\large{\bsifamily{s}}}}$}}}\big(\text{\Large \bsifamily{t}}^{\rightarrow}(\mathcal U)\big),\\
\end{align*}
i.e.  $\text{\Large \bsifamily{t}}$  is $\text{\rsfs{C}}$-continuous.
\end{proof}
As a consequence of corollary(\ref{c-complete}), proposition(\ref{c-ini-cont}) and proposition (\ref{c-unique}) (cf. \cite{AHS}), we obtain 

 \begin{theorem}
The forgetful functor $U: \mathbf{\text{\rsfs{C}}\text{-} CCov}\rightarrow\mathbf{CCov}$ is topological, i.e. the concrete category $\big(\mathbf{\text{\rsfs{C}}\text{-} CCov},\ U\big)$ is topological.
 \end{theorem}
 \subsection{Closed and dense subobjects}
 In this section we introduce a  notion of closed subobjects different from the one allowed in (\cite{CS}).
  
 \begin{defi}
 An subobject $\mathcal T$ of a convergent cover $\mathcal S$ is called
 \begin{itemize}
 \item $\mathbf{\text{\rsfs{C}}}$-closed (in $\mathcal S$) if $  c_{\text{\tiny{$\mathcal S$}}}(\mathcal T)=\mathcal T $;
 \item $\mathbf{\text{\rsfs{C}}}$-dense \big(in $\mathcal S$\big) if $c_{\text{\tiny{$\mathcal S$}}}(\mathcal T)=\mathcal S $.
 \end{itemize}
 \end{defi}
 It is easy to verify thst for the Kuratowski closure operator $K$ of $\mathbf{Top}$, $K$-closed and $K$-dense for a subspace inclusion $M \rightarrowtail X$ means closed and dense in the usual topological sense, respectively.
 
 The $\mathbf{\text{\rsfs{C}}}$-continuity condition (\ref{c-conti}) implies that $\mathbf{\text{\rsfs{C}}}$-closedness is preserve by inverse images,and that $\mathbf{\text{\rsfs{C}}}$-denseness is preserved by images:
 \begin{prop}
Let  $\text{\Large \bsifamily{s}}:\mathcal L\rightarrow \mathcal M $ be a morphism in $\mathbf  CCov$,
 
 \begin{enumerate}
 \item  If $\mathcal V$ is $\mathbf{\text{\rsfs{C}}}$-closed in $\mathcal M$, then $\text{\Large \bsifamily{s}}^{\leftarrow}(\mathcal V)$ is $\mathbf{\text{\rsfs{C}}}$-closed  in $\mathcal L$,
 \item  If $\mathcal U$ is $\mathbf{\text{\rsfs{C}}}$-dense in $\mathcal L$, and $\text{\Large \bsifamily{s}}^{\rightarrow}(\mathcal L)= \mathcal L$, then $\text{\Large \bsifamily{s}}^{\rightarrow}(\mathcal U)$ is $\mathbf{\text{\rsfs{C}}}$-dense in $M$.
 \end{enumerate}
 \end{prop}
 
\begin{proof}\
\begin{enumerate}
\item If $  c_{\text{\tiny{$\mathcal M$}}}(\mathcal V)=\mathcal V $ then  
  $c_{\text{\tiny{$\mathcal L$}}}\big(\text{\Large \bsifamily{s}}^{\leftarrow}(\mathcal V)\big)\subseteq \text{\Large \bsifamily{s}}^{\leftarrow}\big(c_{\text{\tiny{$\mathcal M$}}}(\mathcal V)\big)= \text{\Large \bsifamily{s}}^{\leftarrow}(\mathcal V) $.
\item If $c_{\text{\tiny{$\mathcal L$}}}(\mathcal U) = \mathcal L$  and $\text{\Large \bsifamily{s}}^{\rightarrow}(\mathcal L)= \mathcal L$, then $\mathcal L= \text{\Large \bsifamily{s}}^{\rightarrow}(\mathcal L) = \text{\Large \bsifamily{s}}^{\rightarrow}\big( c_{\text{\tiny{$\mathcal L$}}}(\mathcal U)\big)]\subseteq c_{\text{\tiny{$\mathcal M$}}}(\text{\Large \bsifamily{s}}^{\rightarrow}(\mathcal  U)). $
\end{enumerate}  
\end{proof}

\subsection{A reflective subcategory of $\mathbf{\mathfrak{S}_\text{{\bsifamily{cc}}}(\mathcal L)}$}
For every $\mathcal L$ object of $\mathbf{CCov}$, let $\mathbf{\mathfrak{S}_\text{{\bsifamily{cc}}}(\mathcal L)}^{\mathfrak C}$ denote the  collection of $\mathbf{\text{\rsfs{C}}}$-closed subobjects of $\mathcal L$

Since for every  $\mathcal L\in \mathbf{CCov}$, the inclusion $i: \mathbf{\mathfrak{S}_\text{{\bsifamily{cc}}}(\mathcal L)}^{\mathfrak C}\hookrightarrow \mathbf{\mathfrak{S}_\text{{\bsifamily{cc}}}(\mathcal L)}$ preserves all meets, it has a left Galois adjoint\footnote{We use The Galois Adjunction Theorem in CZF; see \cite{PA}}
\begin{equation}\label{c-clos}
\mathfrak{R}_{_\text{\tiny{$\mathcal L$}}}:\mathbf{\mathfrak{S}_\text{{\bsifamily{cc}}}(\mathcal L)}\rightarrow \ \mathbf{\mathfrak{S}_\text{{\bsifamily{cc}}}(\mathcal L)}^{\mathfrak C}\,\ \text{defined by}\,\ 
\mathfrak{R}_{_\text{\tiny{L}}}(\mathcal T)=\bigcap\{\mathcal V\in\mathbf{\mathfrak{S}_\text{{\bsifamily{cc}}}(\mathcal L)}\mid \mathcal T\subseteq i(\mathcal V)\}.
\end{equation}
\begin{prop}
 The family  $\mathfrak R=(\mathfrak{R}_{_\text{\tiny{$\mathcal L$}}})_{\text{$\mathcal L\in \mathbf{CCov}$}}$ of maps  (\ref{c-clos})  is another closure operator of the category  $\mathbf{CCov}$.
\end{prop}
\begin{proof}
Let $\mathcal L$ be an object of $ \mathbf{CCov}$. Then 
 \begin{itemize}
 \item[($C_1$)] $ \mathcal T\subseteq \mathfrak{R}_{\text{\tiny{$\mathcal L$}}}(\mathcal T) $ for all $\mathcal T \in \mathbf{\mathfrak{S}_\text{{\bsifamily{cc}}}(\mathcal L)}$, because $\mathcal T \subseteq \mathcal V$ for all $\mathcal V\in \mathbf{\mathfrak{S}_\text{{\bsifamily{cc}}}(\mathcal L)}$;
 \item[($C_2$)]  If $\mathcal S\subseteq \mathcal T$ in $\mathbf{\mathfrak{S}_\text{{\bsifamily{cc}}}(\mathcal L)}$, then 
 \begin{align*}
\mathfrak{R}_{_\text{\tiny{$\mathcal L$}}}(\mathcal S)&=\bigcap\{\mathcal V\in \mathbf{\mathfrak{S}_\text{{\bsifamily{cc}}}(\mathcal L)}\mid \mathcal S
\subseteq i(\mathcal V)\}\\ &\subseteq \bigcap\{\mathcal V\in \mathbf{\mathfrak{S}_\text{{\bsifamily{cc}}}(\mathcal L)}\mid \mathcal T\subseteq i(\mathcal V)\} \\ &=\mathfrak{R}_{\text{\tiny{$\mathcal L$}}}(\mathcal T);
 \end{align*}
 \item[($C_3$)] Clearly, we have $\mathfrak{R}_{\text{\tiny{$\mathcal L$}}}(\Phi)=\Phi$.
 \end{itemize}
 Additionally, it is interesting to note that 
 \[
 \mathfrak{R}_{\text{\tiny{$\mathcal L$}}}\big( \mathfrak{R}_{_\text{\tiny{$\mathcal L$}}}(\mathcal T)\big)=\bigcap\{\mathcal V\in \mathbf{\mathfrak{S}_\text{{\bsifamily{cc}}}(\mathcal L)}\mid \mathfrak{R}_{\text{\tiny{$\mathcal L$}}}(\mathcal T)\subseteq i(\mathcal V)\} =\mathfrak{R}_{\text{\tiny{$\mathcal L$}}}(\mathcal T);
 \]
in other words, $\mathfrak R$ is an idempotent closure operator of the category  $\mathbf{CCov}$.
\end{proof}
\begin{coro}
$\mathbf{\mathfrak{S}_\text{{\bsifamily{cc}}}(\mathcal L)}^{\mathfrak C}$  is a reflective subcategory of 
$\mathbf{\mathfrak{S}_\text{{\bsifamily{cc}}}(\mathcal L)}$.
\end{coro}
\begin{proof}
As we have already seen,  for every  $\mathcal L\in \mathbf{CCov}$, the closure map\linebreak $\mathfrak{R}_{_\text{\tiny{$\mathcal L$}}}:\mathbf{\mathfrak{S}_\text{{\bsifamily{cc}}}(\mathcal L)}\rightarrow \mathbf{\mathfrak{S}_\text{{\bsifamily{cc}}}(\mathcal L)}^{\mathfrak C}$ is left adjoint of the inclusion morphism \linebreak $i: \mathbf{\mathfrak{S}_\text{{\bsifamily{cc}}}(\mathcal L)}^{\mathfrak C}\hookrightarrow \mathbf{\mathfrak{S}_\text{{\bsifamily{cc}}}(\mathcal L)}$
\end{proof}

\section{Interior Operators}
We shall be conserned in this section with a version on the category  $\mathbf CCov$ of the interior operator studied in \cite{LO}. 

 \begin{defi}
 An interior operator $I$ of the category   $\mathbf CCov$ is given by a family $I =(i_{\text{\tiny{$\mathcal S$}}})_{\text{$\mathcal S\in \mathbf CCov$}}$ of maps $i_{\text{\tiny{$\mathcal S$}}}:\mathfrak{S}_\text{{\bsifamily{cc}}}(\mathcal S)\rightarrow \mathfrak{S}_\text{{\bsifamily{cc}}}(\mathcal S)$ such that

 \begin{itemize}
 \item[($I_1)$] $\left(\text{Contraction}\right)$\,\  $i_{\text{\tiny{$\mathcal S$}}}(\mathcal T)\subseteq \mathcal T$ for all $\mathcal T \in \mathfrak{S}_\text{{\bsifamily{cc}}}(\mathcal S)$;
 \item[($I_2)$] $\left(\text{Monotonicity}\right)$\,\  If $\mathcal R\subseteq \mathcal T$ in $\mathfrak{S}_\text{{\bsifamily{cc}}}(\mathcal S)$, then $i_{\text{\tiny{$\mathcal S$}}}(\mathcal R)\subseteq i_{\text{\tiny{$\mathcal S$}}}(\mathcal T)$
 \item[($I_3)$] $\left(\text{Upper bound}\right)$\,\  $i_{\text{\tiny{$\mathcal S$}}}(\mathcal S)=\mathcal S$.
 \end{itemize}
 
  \end{defi}
 \begin{defi}
 An $I$-space is a pair $(\mathcal S, i_{\text{\tiny{$\mathcal S$}}})$ where $\mathcal S$ is an object of  $\mathbf CCov$   and $ i_{\text{\tiny{$\mathcal S$}}}$ is an interior operator on $\mathcal S$.
 \end{defi}
 \begin{remark}
 If $\text{\Large \bsifamily{s}} \subseteq S_1 \times S_2$ is a representative of   an equivalence class of relations between $S_1$ and $S_2$ which respect the covers, we donte by
 \begin{itemize}
\item  $\text{\Large \bsifamily{s}}^{-1}$ the inverse of $\text{\Large \bsifamily{s}}$;
\item $\text{\Large \bsifamily{s}}^{\rightarrow}$ the direct image of $\text{\Large \bsifamily{s}}$ defined by $$(\forall X\subset S_{1})( \text{\Large \bsifamily{s}}^{\rightarrow}(X)=\big\{t\in S_{2}\mid \big (\exists s\in X\big)\big( (s,t)\in \text{\Large \bsifamily{s}}\big)\big\}; $$
\item $\text{\Large \bsifamily{s}}^{\leftarrow}$ the inverse  image of $\text{\Large \bsifamily{s}}$ is the direct image of $\text{\Large \bsifamily{s}}^{-1}$.
 \end{itemize}
 \end{remark}
 
 \begin{defi}
 A morphism $\text{\Large \bsifamily{s}}: \mathcal L\rightarrow \mathcal M$ of $\mathbf CCov$ is said to be $I$-continuous if 
 \begin{equation}\label{i-conti}
 \text{\Large \bsifamily{s}}^{\leftarrow}\big( i_{\text{\tiny{$\mathcal M$}}}(\mathcal T)\big) \subseteq i_{\text{\tiny{$\mathcal L$}}}\big(  \text{\Large \bsifamily{s}}^{\leftarrow}(\mathcal T)\big)
 \end{equation}
 for all $\mathcal T \in \mathfrak{S}_\text{{\bsifamily{cc}}}(\mathcal S)$.
 \end{defi}
 
 \begin{prop}
 Let $\text{\Large \bsifamily{s}}:\mathcal L\rightarrow \mathcal M$ and $\text{\Large \bsifamily{t}}:\mathcal M\rightarrow \mathcal N$ be two $I$-continuous morphisms  of   $\mathbf CCov$ then $\text{\Large \bsifamily{t}}\centerdot \text{\Large \bsifamily{s}}$ is an $I$-continuous morphism of   $\mathbf{CCov}$.
 \end{prop}
 \begin{proof}
 Since $\text{\Large \bsifamily{t}}:\mathcal M\rightarrow \mathcal N$ is $I$-continuous, we have  $$\text{\Large \bsifamily{t}}^{\leftarrow}\big( i_{\text{\tiny{$\mathcal N$}}}(\mathcal W)\big)\subseteq i_{\text{\tiny{$\mathcal M$}}}\big( \text{\Large \bsifamily{t}}^{\leftarrow}(\mathcal W)\big)$$
  for all $\mathcal T \in \mathfrak{S}_\text{{\bsifamily{cc}}}(\mathcal N)$, it follows that
   $$\text{\Large \bsifamily{s}}^{\leftarrow}\Big(\text{\Large \bsifamily{t}}^{\leftarrow}\big(( i_{\text{\tiny{$\mathcal N$}}}(\mathcal W)\big)\Big)\subseteq \text{\Large \bsifamily{s}}^{\leftarrow}\Big( i_{\text{\tiny{$\mathcal M$}}}\big( \text{\Large \bsifamily{t}}^{\leftarrow}(\mathcal W)\big)\Big);$$
  now,  by the $I$-continuity of $\text{\Large \bsifamily{s}}$,
 $$
\text{\Large \bsifamily{s}}^{\leftarrow}\Big( i_{\text{\tiny{$\mathcal M$}}} \big(\text{\Large \bsifamily{t}}^{\leftarrow}(\mathcal W)\big)\Big)\subseteq i_{\text{\tiny{$\mathcal L$}}}\Big( \text{\Large \bsifamily{s}}^{\leftarrow}\big(\text{\Large \bsifamily{t}}^{\leftarrow}(\mathcal W)\big)\Big),
 $$
  
therefore $$\text{\Large \bsifamily{s}}^{\leftarrow}\Big[\text{\Large \bsifamily{t}}^{\leftarrow}\big( i_{\text{\tiny{$\mathcal N$}}}(\mathcal W)\big)\Big)\subseteq i_{\text{\tiny{$\mathcal L$}}}\Big( \text{\Large \bsifamily{s}}^{\leftarrow}\big[\text{\Large \bsifamily{t}}^{\leftarrow}(\mathcal W)\big)\Big),$$
  that is to say 
  $$(\text{\Large \bsifamily{t$\centerdot$ s}})^{\leftarrow}\big( i_{\text{\tiny{$\mathcal N$}}}(\mathcal W)\Big)\subseteq i_{\text{\tiny{$\mathcal L$}}}\Big( (\text{\Large \bsifamily{t$\centerdot$ s}})^{\leftarrow}(\mathcal W)\Big)$$
\end{proof}
  As a consequence we obtain
 \begin{defi}
 The category $\mathbf{I}\text{-}\mathbf CCov$ of $I$-spaces comprises the following data:
 \begin{enumerate}
 \item {\bf Objects}: Pairs $(\mathcal S, i_{\text{\tiny{$\mathcal S$}}})$ where $\mathcal S$ is an object of  $\mathbf CCov$   and $ i_{\text{\tiny{$\mathcal S$}}}$ is an interior operator on $\mathcal S$.
\item {\bf Morphisms}: Morphisms of $\mathbf CCov$ which are $I$-continuous.
 \end{enumerate}
 \end{defi}

 \subsection{The lattice structure of all interior operators}
 For the category $\mathbf CCov$ we consider the collection
 \[
 Int(\mathbf{\mathbf CCov})
 \]
 of all  interior operators on $\mathbf CCov$. It is ordered by
\[
 I\leqslant J \Leftrightarrow i_{\text{\tiny{$\mathcal S$}}}(\mathcal T)\subseteq j_{\text{\tiny{$\mathcal S$}}}(\mathcal T), \,\,\ \text{for all $\mathcal T \in \mathfrak{S}_\text{{\bsifamily{cc}}}(\mathcal S)$  and all $\mathcal S$  object of $\mathbf CCov$. }
\]

 This way $Int(\mathbf{\mathbf CCov})$ inherits a lattice structure from $\mathfrak{S}_\text{{\bsifamily{cc}}}(\mathcal S)$:
 
 \begin{prop}
 Every family $(I_{\text{\tiny{$\lambda$}}})_{\text{\tiny{$\lambda\in \Lambda$}}}$ in $ Int(\mathbf{\mathbf CCov})$ has a join $\bigvee\limits_{\text{\tiny{$\lambda\in \Lambda $}}}I_{\text{\tiny{$\lambda $}}}$ and a meet $\bigwedge\limits_{\text{\tiny{$\lambda\in \Lambda $}}}I_{\text{\tiny{$\lambda $}}}$ in $ Int(\mathbf{\mathbf CCov})$. The discrete interior operator
 \[ I_{\text{\tiny{$D$}}}=({i_{\text{\tiny{$D$}}}} _{\text{\tiny{$\mathcal S$}}})_{\text{$\mathcal S\in \mathbf CCov$}}\,\,\ \text{with}\,\,\
 {i_{\text{\tiny{$D$}}}}_{\text{\tiny{$\mathcal S$}}}(\mathcal T)=
 \mathcal T\,\,\ \text{for all} \,\ \mathcal T \in {\mathfrak S}_\text{{\bsifamily{cc}}}(\mathcal S)
 \]
  is the largest element in $Int(\mathbf CCov)$, and the trivial interior operator
  \[ 
I_{\text{\tiny{$ T$}}}=({i_{\text{\tiny{$ T$}}}}_\text{\tiny{$\mathcal S$}})_{\text{$\mathcal S\in \mathbf CCov$}}\,\,\ \text{with}\,\,\ {i_{\text{\tiny{$T$}}}}_{\text{\tiny{$\mathcal S$}}}(\mathcal T)=
  \begin{cases}
\Phi& \text{for all}\,\ \mathcal T\in {\mathfrak S}_\text{{\bsifamily{cc}}}(\mathcal S)\\
\mathcal S&\text {if}\,\ \mathcal T=\mathcal S
\end{cases}
 \]
  is the least one.
 \end{prop}

 \begin{proof}
For $\Lambda\ne\emptyset$, let $\widehat{I}=\bigvee\limits_{\text{\tiny{$\lambda\in\Lambda $}}}I_{\text{\tiny{$\lambda $}}}$, then 
 \[
 \widehat{i_{\text{\tiny{$\mathcal S$}}}}=\bigvee\limits_{\text{\tiny{$\lambda\in \Lambda$}}} {i_{\text{\tiny{$\lambda $}}}}_{\text{\tiny{$\mathcal S$}}},
 \]
 for all  $\mathcal S$ object of $\mathbf CCov$, satisfies
\begin{itemize}
 \item $ \widehat{i_{\text{\tiny{$\mathcal S$}}}}(\mathcal T)\subseteq \mathcal T$,  because ${i_{\text{\tiny{$\lambda $}}}}_{\text{\tiny{$\mathcal S$}}}(\mathcal T)\subseteq \mathcal T$ for all $\mathcal T\in {\mathfrak S}_\text{{\bsifamily{cc}}}$ and for all $\lambda \in \Lambda$.
\item If $\mathcal R\leqslant \mathcal T$ in ${\mathfrak S}_\text{{\bsifamily{cc}}}$ then ${i_{\text{\tiny{$\lambda $}}}}_{\text{\tiny{$\mathcal S$}}}(\mathcal R)\subseteq {i_{\text{\tiny{$\lambda $}}}}_{\text{\tiny{$\mathcal S$}}}(\mathcal T)$  for all $\lambda \in \Lambda$, therefore $ \widehat{i_{\text{\tiny{$\mathcal S$}}}}(\mathcal R)\subseteq \widehat{i_{\text{\tiny{$\mathcal S$}}}}(\mathcal T)$.
 \item Since ${i_{\text{\tiny{$\lambda $}}}}_{\text{\tiny{$\mathcal S$}}}(\mathcal S)=\mathcal S $  for all $\lambda \in \Lambda$, we have that  $ \widehat{i_{\text{\tiny{$\mathcal S$}}}}(\mathcal S)=\mathcal S$.
 \end{itemize}
Similarly  $\bigwedge\limits_{\text{\tiny{$\lambda\in \Lambda $}}}I_{\text{\tiny{$\lambda $}}}$,\,\  $ I_{\text{\tiny{$D$}}}$ and $I_{\text{\tiny{$T$}}}$ are interior operators.
 \end{proof}
 \begin{coro}\label{i-complete}
  For  every  object $\mathcal S$ of $\mathbf CCov$
  \[
  Int(\mathcal S) = \{i_{\text{\tiny{$\mathcal S$}}}\mid i_{\text{\tiny{$\mathcal S$}}}\,\ \text{ is an interior operator on}\,\ \mathcal S\}
  \]
  is a complete lattice.
 \end{coro}

 \subsection{Initial interior operators}
Let  $\mathbf{I}\text{-}\mathbf CCov$ be the category of $I$-spaces. Let  $(M,i_{\text{\tiny{$\mathcal M$}}})$ be an object of  $\mathbf{I}\text{-}\mathbf CCov$ and let $\mathcal L$ be an object of $\mathbf CCov$. For each morphism  $\text{\large \bsifamily{s}}:\mathcal L\rightarrow \mathbf M$ in $\mathbf CCov$ we define on $\mathcal L$ the operator 
\vspace{0.5cm}

\begin{center}
\begin{tikzpicture}[scale=0.8]
\node (A) at (0,3) {$\mathfrak{S}_\text{{\bsifamily{cc}}}(\mathcal L)$};
\node (B) at (3,3) {$\mathfrak{S}_\text{{\bsifamily{cc}}}(\mathcal M)$};
\node (C) at (0,0) {$\mathfrak{S}_\text{{\bsifamily{cc}}}(\mathcal L)$};
\node (D) at (3,0) {$\mathfrak{S}_\text{{\bsifamily{cc}}}(\mathcal M)$};
\draw[->] (A) -- (B);
\draw[->, dashed] (A) -- (C);
\draw[->] (B) -- (D);
\draw[->] (D) -- (C);
\node at (1.5,3.3) {$ \text{\large \bsifamily{s}}^{\rightarrow} $};
\node at (1.5,-0.3) {$ \text{\large \bsifamily{s}}^{\leftarrow} $};
\node at (3.4,1.5) {$ i_{\text{\tiny{$\mathcal M$}}} $};
\node at (-0.4,1.5) {$ i_{\text{\tiny{${\mathcal L}_{ \text{\tiny \bsifamily{s}}}$}}}$};
\end{tikzpicture}
\end{center}

\begin{equation} \label{i-initial}
 i_{\text{\tiny{${\mathcal L}_{ \text{\tiny \bsifamily{s}}}$}}}:=\text{\large \bsifamily{s}}^{\leftarrow} \centerdot i_{\text{\tiny{$\mathcal M$}}}\centerdot \text{\large \bsifamily{s}}^{\rightarrow} .
\end{equation}
\begin{prop}\label{inic-cont}
The operator (\ref{i-initial}) is an interior operator on $\mathcal L$ for which the morphism $\text{\large \bsifamily{s}}$ is $I$-continuous.
\end{prop}
\begin{proof}\
\begin{enumerate}
\item[($I_1)$] $\left(\text{Contraction}\right)$\,\  $i_{\text{\tiny{${\mathcal L}_{ \text{\tiny \bsifamily{s}}}$}}}(\mathcal T)= \text{\large \bsifamily{s}}^{\leftarrow} \centerdot i_{\text{\tiny{$\mathcal M$}}}\centerdot \text{\large \bsifamily{s}}^{\rightarrow}(\mathcal T)\subseteq \text{\large \bsifamily{s}}^{\leftarrow} \centerdot  \text{\large \bsifamily{s}}^{\rightarrow}(\mathcal T)\subseteq \mathcal T$ for all $\mathcal T \in \mathfrak{S}_\text{{\bsifamily{cc}}}(\mathcal S)$ ;

 \item[($I_2)$] $\left(\text{Monotonicity}\right)$\,\   $\mathcal R\subseteq \mathcal T$ in $\mathfrak{S}_\text{{\bsifamily{cc}}}(\mathcal S)$, implies $\text{\large \bsifamily{s}}^{\rightarrow}(\mathcal R)\subseteq\text{\large \bsifamily{s}}^{\rightarrow}(\mathcal T)$, then $i_{\text{\tiny{$\mathcal M$}}}\centerdot \text{\large \bsifamily{s}}^{\rightarrow}(\mathcal R)\subseteq i_{\text{\tiny{$\mathcal M$}}}\centerdot \text{\large \bsifamily{s}}^{\rightarrow}(\mathcal T)$, consequently  $ \text{\large \bsifamily{s}}^{\leftarrow} \centerdot i_{\text{\tiny{$\mathcal M$}}}\centerdot \text{\large \bsifamily{s}}^{\rightarrow}(\mathcal R)\subseteq \text{\large \bsifamily{s}}^{\leftarrow} \centerdot i_{\text{\tiny{$\mathcal M$}}}\centerdot \text{\large \bsifamily{s}}^{\rightarrow}(\mathcal T)$;
 
 \item[($I_3)$] $\left(\text{Upper bound}\right)$\,\  $i_{\text{\tiny{${\mathcal L}_{ \text{\tiny \bsifamily{s}}}$}}}(\mathcal S)=\text{\large \bsifamily{s}}^{\leftarrow} \centerdot i_{\text{\tiny{$\mathcal M$}}}\centerdot \text{\large \bsifamily{s}}^{\rightarrow}(\mathcal S)=\mathcal S$.
\end{enumerate}
Finally,
\begin{align*}
\text{\large \bsifamily{s}}^{\leftarrow} \big( i_{\text{\tiny{$\mathcal M$}}}
(\mathcal T)\big)&\subseteq\text{\large \bsifamily{s}}^{\leftarrow} \big( i_{\text{\tiny{$\mathcal M$}}}\centerdot \text{\large \bsifamily{s}}^{\rightarrow}\centerdot \text{\large \bsifamily{s}}^{\leftarrow}(\mathcal T)\big)\\
&=\big(\text{\large \bsifamily{s}}^{\leftarrow} \centerdot i_{\text{\tiny{$\mathcal M$}}}\centerdot \text{\large \bsifamily{s}}^{\rightarrow}\big)\centerdot \text{\large \bsifamily{s}}^{\leftarrow}(\mathcal T)\\
&= i_{\text{\tiny{$L_{f}$}}}\big(f^{-1}(T)\big),
\end{align*}
 for all $\mathcal T\in {\mathfrak S}_\text{{\bsifamily{cc}}}(\mathcal S)$.
\end{proof}

It is clear that $  i_{\text{\tiny{${\mathcal L}_{ \text{\tiny \bsifamily{s}}}$}}}$ is the coarsest interior operator on $\mathcal L$ for which the morphism $\text{\large \bsifamily{s}}$ is $I$-continuous; more precisaly

\begin{prop}\label{i-unique}
Let $(\mathcal N,i_{\text{\tiny{$\mathcal N$}}})$ and $(\mathcal M,i_{\text{\tiny{$\mathcal M$}}})$ be objects of  $\mathbf{I}\text{-}\mathbf CCov$ , and let $\mathcal L$ be an object of  $\mathbf CCov$. For each morphism  $\text{\large \bsifamily{t}}:\mathcal N\rightarrow \mathcal L$ in  $\mathbf CCov$ and for $\text{\large \bsifamily{s}}:(\mathcal L,i_{\text{\tiny{${\mathcal L}_{ \text{\tiny \bsifamily{s}}}$}}})\rightarrow (\mathcal M,i_{\text{\tiny{$\mathcal M$}}})$ an $I$-continuous morphism, $\text{\large \bsifamily{t}}$  is $I$-continuous if and only if $\text{\large \bsifamily{s}}\centerdot \text{\large \bsifamily{t}}$ is $I$-continuous.
\end{prop}
\begin{proof}
Suppose that $\text{\large \bsifamily{s}}\centerdot \text{\large \bsifamily{t}}$ is $I$-continuous, i. e.
$$(\text{\large \bsifamily{s}}\centerdot \text{\large \bsifamily{t}})^{\leftarrow}\big(i_{\text{\tiny{$\mathcal M$}}}(\mathcal T)\big)\subseteq i_{\text{\tiny{$\mathcal N$}}}\big( (\text{\large \bsifamily{s}}\centerdot \text{\large \bsifamily{t}})^{\leftarrow}(\mathcal T) \big)$$
 for all$\mathcal T\in {\mathfrak S}_\text{{\bsifamily{cc}}}(\mathcal M)$. Then, for all $\mathcal R\in {\mathfrak S}_\text{{\bsifamily{cc}}}(\mathcal L)$, we have
\begin{align*}
\text{\large \bsifamily{t}}^{\leftarrow}\big(i_{\text{\tiny{${\mathcal L}_{ \text{\tiny \bsifamily{s}}}$}}}(\mathcal R)\big)&=\text{\large \bsifamily{t}}^{\leftarrow}\big(\text{\large \bsifamily{s}}^{\leftarrow} \centerdot i_{\text{\tiny{$\mathcal M$}}}\centerdot \text{\large \bsifamily{s}}^{\rightarrow}(\mathcal R)\big)=(\text{\large \bsifamily{s}}\centerdot \text{\large \bsifamily{t}})^{\leftarrow}\big( i_{\text{\tiny{$\mathcal M$}}}(\text{\large \bsifamily{s}}^{\rightarrow}(\mathcal R)) \big)\\
 &\subseteq i_{\text{\tiny{$\mathcal N$}}}\big((\text{\large \bsifamily{s}}\centerdot \text{\large \bsifamily{t}})^{\leftarrow}\big( \text{\large \bsifamily{s}}^{\rightarrow}(\mathcal R)\big) \big)=i_{\text{\tiny{$\mathcal N$}}}\big( \text{\large \bsifamily{t}}^{\leftarrow}\centerdot \text{\large \bsifamily{s}}^{\leftarrow}\centerdot  \text{\large \bsifamily{s}}^{\rightarrow} (\mathcal R)\big)\\
 &\subseteq i_{\text{\tiny{$\mathcal N$}}}\big(\text{\large \bsifamily{t}}^{\leftarrow} (S)\big),\\
\end{align*}
i.e.  $\text{\large \bsifamily{t}} $  is $I$-continuous.
\end{proof}
As a consequence of corollary(\ref{i-complete}), proposition(\ref{inic-cont}) and proposition (\ref{i-unique}) (cf. \cite{AHS}), we obtain 
 \begin{theorem}
The forgetful functor $U:\mathbf{I}\text{-}\mathbf CCov\rightarrow \mathbf CCov$ is topological, i.e. the concrete category $\big(\mathbf{I}\text{-}\mathbf CCov,\ U\big)$ is topological.
 \end{theorem}
\subsection{Open subobjects}
 In this section we introduce a  notion of open subobjects different from the one allowed in (\cite{CS}).

\begin{defi}
 An subobject $\mathcal T$ of a convergent cover $\mathcal S$ is called
  $\mathbf{\text{\rsfs{I}}}$-open (in $\mathcal S$) if $  i_{\text{\tiny{$\mathcal S$}}}(\mathcal T)=\mathcal T $;
  \end{defi}
 It is easy to verify thst for the Kuratowski interior operator $i$ of $\mathbf{Top}$, $i$-open for a subspace inclusion $M \rightarrowtail X$ means open in the usual topological sense.
 
 The $\mathbf{\text{\rsfs{I}}}$-continuity condition (\ref{c-conti}) implies that $\mathbf{\text{\rsfs{I}}}$-openness is preserve by inverse images:
 
 \begin{prop}
Let  $\text{\Large \bsifamily{s}}:\mathcal L\rightarrow \mathcal M $ be a morphism in $\mathbf  CCov$.   If $\mathcal V$ is $\mathbf{\text{\rsfs{I}}}$-open in $\mathcal M$, then $\text{\Large \bsifamily{s}}^{\leftarrow}(\mathcal V)$ is $\mathbf{\text{\rsfs{I}}}$-open  in $\mathcal L$,
 \end{prop}

\begin{proof}

If $  i_{\text{\tiny{$\mathcal M$}}}(\mathcal V)=\mathcal V $ then  
  $
  \text{\Large \bsifamily{s}}^{\leftarrow}\big(\mathcal V)\big) =
 \text{\Large \bsifamily{s}}^{\leftarrow}\big( i_{\text{\tiny{$\mathcal M$}}}(\mathcal V)\big) \subseteq i_{\text{\tiny{$\mathcal L$}}}\big(  \text{\Large \bsifamily{s}}^{\leftarrow}(\mathcal V)\big).
$
In other words,  $\text{\Large \bsifamily{s}}^{\leftarrow}(\mathcal V)$ is $\mathbf{\text{\rsfs{I}}}$-open  in $\mathcal L$.
\end{proof}

\subsection{A coreflective subcategory of $\mathbf{\mathfrak{S}_\text{{\bsifamily{cc}}}(\mathcal L)}$ }
For every $\mathcal L\in \mathbf{CCov}$, let $\mathbf{\mathfrak{S}_\text{{\bsifamily{cc}}}(\mathcal L)}^{\mathfrak O}$ denote the  collection of $\mathbf{\text{\rsfs{I}}}$-open subobjects of $\mathcal L$

Since for every  $\mathcal L\in \mathbf{CCov}$, the inclusion $j: \mathbf{\mathfrak{S}_\text{{\bsifamily{cc}}}(\mathcal L)}^{\mathfrak O}\hookrightarrow \mathbf{\mathfrak{S}_\text{{\bsifamily{cc}}}(\mathcal L)}$ preserves all joins, it has a right Galois adjoint\footnote{We use The Galois Adjunction Theorem in CZF; see \cite{PA}}
\begin{equation}\label{i-clos}
\mathfrak{K}_{_\text{\tiny{$\mathcal L$}}}:\mathbf{\mathfrak{S}_\text{{\bsifamily{cc}}}(\mathcal L)}\rightarrow \ \mathbf{\mathfrak{S}_\text{{\bsifamily{cc}}}(\mathcal L)}^{\mathfrak O}\,\ \text{defined by}\,\ 
\mathfrak{K}_{_\text{\tiny{L}}}(\mathcal T)=\bigcup\{\mathcal V\in\mathbf{\mathfrak{S}_\text{{\bsifamily{cc}}}^{\mathfrak O}(\mathcal L)}\mid  j(\mathcal V) \subseteq \mathcal T\}.
\end{equation}
\begin{prop}
 The family  $\mathfrak K=(\mathfrak{K}_{_\text{\tiny{$\mathcal L$}}})_{\text{$\mathcal L\in \mathbf{CCov}$}}$ of maps  (\ref{i-clos})  is another interior operator of the category  $\mathbf{CCov}$.
\end{prop}

\begin{proof}
Let $\mathcal L$ be an object of $ \mathbf{CCov}$. Then 
 \begin{itemize}
 \item[($I_1$)] $ \mathfrak{K}_{\text{\tiny{$\mathcal L$}}}(\mathcal T)\subseteq \mathcal T $ for all $\mathcal T \in \mathbf{\mathfrak{S}_\text{{\bsifamily{cc}}}(\mathcal L)}$, because $\mathcal V \subseteq \mathcal T$ for some $\mathcal V\in \mathbf{\mathfrak{S}_\text{{\bsifamily{cc}}}(\mathcal L)}^{\mathfrak O}$;
 \item[($I_2$)]  If $\mathcal S\subseteq \mathcal T$ in $\mathbf{\mathfrak{S}_\text{{\bsifamily{cc}}}(\mathcal L)}$, then 
 \begin{align*}
\mathfrak{K}_{_\text{\tiny{$\mathcal L$}}}(\mathcal S)&=\bigcup\{\mathcal V\in \mathbf{\mathfrak{S}_\text{{\bsifamily{cc}}}(\mathcal L)}\mid  j(\mathcal V)\subseteq S\}\\ &\subseteq \bigcup\{\mathcal V\in \mathbf{\mathfrak{S}_\text{{\bsifamily{cc}}}(\mathcal L)}\mid j(\mathcal V)\subseteq \mathcal T\} \\ &=\mathfrak{K}_{\text{\tiny{$\mathcal L$}}}(\mathcal T);
 \end{align*}
 \item[($I_3$)] Clearly, we have \  $\mathfrak{K}_{_\text{\tiny{$\mathcal L$}}}(\mathcal L)=\mathcal L$.
 \end{itemize}
 
 Additionally, it is interesting to note that 
 \[
 \mathfrak{K}_{\text{\tiny{$\mathcal L$}}}\big( \mathfrak{K}_{_\text{\tiny{$\mathcal L$}}}(\mathcal T)\big)=\bigcup\{\mathcal V\in \mathbf{\mathfrak{S}_\text{{\bsifamily{cc}}}(\mathcal L)}\mid \mathfrak{K}_{\text{\tiny{$\mathcal L$}}}(\mathcal T)\subseteq j(\mathcal V)\} =\mathfrak{R}_{\text{\tiny{$\mathcal L$}}}(\mathcal T);
 \]
in other words, $\mathfrak G$ is an idempotent interior operator of the category  $\mathbf{CCov}$.
\end{proof}
\begin{coro}
$\mathbf{\mathfrak{S}_\text{{\bsifamily{cc}}}(\mathcal L)}^{\mathfrak O}$  is a reflective subcategory of 
$\mathbf{\mathfrak{S}_\text{{\bsifamily{cc}}}(\mathcal L)}$.
\end{coro}
\begin{proof}
As we have already seen,  for every  $\mathcal L\in \mathbf{CCov}$, the interior map\linebreak $\mathfrak{K}_{_\text{\tiny{$\mathcal L$}}}:\mathbf{\mathfrak{S}_\text{{\bsifamily{cc}}}(\mathcal L)}\rightarrow \mathbf{\mathfrak{S}_\text{{\bsifamily{cc}}}(\mathcal L)}^{\mathfrak O}$ is left adjoint of the inclusion morphism \linebreak $j: \mathbf{\mathfrak{S}_\text{{\bsifamily{cc}}}(\mathcal L)}^{\mathfrak O}\hookrightarrow \mathbf{\mathfrak{S}_\text{{\bsifamily{cc}}}(\mathcal L)}$
\end{proof}

\end{document}